\documentclass[11pt]{amsart}

\usepackage{pgfplots}
\pgfplotsset{compat=1.18}
\usepgfplotslibrary{fillbetween}

\usepackage{subcaption}
\usepackage{amsmath, amsthm, amssymb, enumerate, hyperref, xcolor, tikz-cd, listings,hyperref, MnSymbol, textcomp}
\usepackage{tikz}
\usepackage{float}

\usepackage{fullpage}

\newcommand{\makeheading}[1]%
{\hspace*{-\marginparsep minus \marginparwidth}%
	\begin{minipage}[t]{\textwidth}%
		{\large \bfseries #1}\\[-0.15\baselineskip]%
		\rule{\columnwidth}{1pt}%
\end{minipage}}

\newcommand{\MLS}{\text{MLS}}

\usepackage{mathtools}

\usepackage{amsmath, amsthm, amssymb, enumerate, hyperref, xcolor, tikz-cd, listings,hyperref,comment}

\usepackage[tt=false, type1=true]{libertine}
\usepackage[varqu]{zi4}
\usepackage[libertine]{newtxmath}
\usepackage[T1]{fontenc}
\usepackage{parskip}

\newcommand{\im}{\text{Im}}

\newcommand{\R}{\mathbb{R}}

\newcommand{\F}{\mathcal{F}}

\newcommand{\Vol}{\text{Vol}}

\newcommand{\supp}{\text{supp}}

\usepackage{soul}
\usepackage{cancel}

\theoremstyle{plain}
\newtheorem{theorem}{Theorem}
\newtheorem{claim}[theorem]{Claim}

\newtheorem{lemma}[theorem]{Lemma}

\theoremstyle{definition}
\newtheorem{rem}[theorem]{Remark}

\renewcommand{\phi}{\varphi}

\renewcommand{\epsilon}{\varepsilon}

\def\semicolon{;}
\def\applytolist#1{
	\expandafter\def\csname multi#1\endcsname##1{
		\def\multiack{##1}\ifx\multiack\semicolon
		\def\next{\relax}
		\else
		\csname #1\endcsname{##1}
		\def\next{\csname multi#1\endcsname}
		\fi
		\next}
	\csname multi#1\endcsname}

\def\calc#1{\expandafter\def\csname c#1\endcsname{{\mathcal #1}}}
\applytolist{calc}QWERTYUIOPLKJHGFDSAZXCVBNM;
\def\bbc#1{\expandafter\def\csname bb#1\endcsname{{\mathbb #1}}}
\applytolist{bbc}QWERTYUIOPLKJHGFDSAZXCVBNM;
\def\bfc#1{\expandafter\def\csname bf#1\endcsname{{\mathbf #1}}}
\applytolist{bfc}QWERTYUIOPLKJHGFDSAZXCVBNM;
\def\sfc#1{\expandafter\def\csname s#1\endcsname{{\sf #1}}}
\applytolist{sfc}QWERTYUIOPLKJHGFDSAZXCVBNM;
\def\fc#1{\expandafter\def\csname f#1\endcsname{{\mathfrak #1}}}
\applytolist{fc}QWERTYUIOPLKJHGFDSAZXCVBNM;

\newcommand{\short}{\text{short}}
\title{A counterexample to marked length spectrum semi-rigidity}

\author{Andrey Gogolev and James Marshall Reber}
\thanks{The authors were partially supported by NSF DMS-2247747}
\begin{document}
	
	\begin{abstract}
		Given a closed, orientable, negatively curved Riemannian surface $(M,g)$, we show how to construct a perturbation $(M,g^\prime)$ such that each closed geodesic becomes longer, and yet there is no diffeomorphism ${f \colon (M,g^\prime)\to (M,g)}$ which contracts every tangent vector.
	\end{abstract}
	\maketitle
	\section{Introduction} \label{sec:introduction}

	Let $M$ be a closed, connected manifold, and let $g$ be a Riemannian metric on $M$ with everywhere negative sectional curvature. It is well known that inside of every non-trivial free homotopy $\sigma$ there exists a unique closed geodesic $\gamma_\sigma$ \cite[Chapter 12]{doCarmo}.
	Denoting the length of a curve $\gamma$ with respect to the metric $g$ by $\ell_g(\gamma)$, the \emph{marked length spectrum} $\MLS_g$ is the function which takes a free homotopy class $\sigma$ and returns the length of the unique closed geodesic $\gamma_\sigma$ in the free homotopy class.
	
	The marked length spectrum has attracted a lot of attention due to the \emph{marked length spectrum rigidity conjecture}, which states that if $\MLS_g( \sigma) = \MLS_{g'}(\sigma)$ for all free homotopy classes $\sigma$, then there is a diffeomorphism $f : M \rightarrow M$ such that $f^*(g') = g$ \cite[Conjecture 3.1]{BK}. The conjecture is known to hold in dimension two \cite{croke, otal}, and in dimensions three and higher when the manifold is locally symmetric \cite{BCG, hamenstadt}, when the metrics are sufficiently close in a fine topology \cite{gl}, or when the metrics are conformally equivalent \cite[Theorem 2]{katok}. The conjecture is still open in general.
	
	Variations of the marked length spectrum rigidity conjecture have been considered in recent years. For example, Butt considered in \cite{butt} an ``approximate version'' of marked length spectrum rigidity, where the marked length spectra being ``approximately equal'' implies that the metrics are ``approximately isometric.'' Another example can be found in \cite{CDS}, where it is conjectured that the inequality $\MLS_g \leq \MLS_{g'}$ implies $\Vol(g) \leq \Vol(g')$, with $\Vol(g) = \Vol(g')$ holding if and only if $g$ is isometric to $g'$.
	Similar to the marked length spectrum conjecture, this is known to be true if $M$ is a surface \cite[Theorem 1.1]{CD}, and in dimensions three and higher if the metrics are sufficiently close in a fine topology \cite{gl} or if the metrics are conformally equivalent \cite[Theorem 1.2]{CD}. Both of these variations can be seen as loosening the equality assumption in the marked length spectrum conjecture and asking what geometric information remains.
	
	In the same vein as above, we wish to explore two marked length spectrum ``semi-rigidity'' type problems, i.e., problems which relate an inequality on the marked length spectrum of two metrics to the existence of a diffeomorphism on $M$ carrying some geometric information. For the first problem, we consider whether or not an inequality on the marked length spectrum implies the existence of a \emph{volume shrinking diffeomorphism} between $g$ and $g'$, i.e., a diffeomorphism whose Jacobian is bounded above by one. In the cases where it is known that $\MLS_g \leq \MLS_{g'}$ implies $\Vol(g) \leq \Vol(g')$, it can be deduced using \cite[Theorem 1.1]{CD} and Moser's homotopy trick \cite{Moser} that if $g$ and $g'$ are two negatively curved metrics on $M$ with $\MLS_g \leq \MLS_{g'}$, then there exists a volume shrinking diffeomorphism $f : M \rightarrow M$. 
	
	For the second problem, we consider whether or not an inequality on the marked length spectrum implies the existence of a length \emph{shrinking diffeomorphism} between $g$ and $g'$, that is, a diffeomorphism $f : M \rightarrow M$ such that $\|D_xf(v)\| \leq \|v\|'$ for all $(x,v) \in TM.$
	Somewhat surprisingly, our main result shows that the answer is ``no'' in a rather strong sense.
	To help with notation, we denote a shrinking diffeomorphism by $f : (M,g') \rightarrow (M,g)$ to help indicate which metric is ``shrinking'' under the diffeomorphism.
	
	\begin{theorem} \label{thm:perturb}
		Let $M$ be a closed, connected, orientable surface, and let $g$ be a negatively curved metric on $M$. Then $g$ admits arbitrarily $C^\infty$-small perturbations $g'$ for which there exists $\epsilon > 0$ so that we have \hbox{$\MLS_{g'} > (1+\epsilon)\MLS_g$} and there does not exist a shrinking diffeomorphism ${f : (M,g') \rightarrow (M,g)}$.
	\end{theorem}

	\begin{rem} We make two observations regarding Theorem \ref{thm:perturb}.
		\begin{enumerate}
			\item Although the result is stated for Riemannian metrics, we note that our arguments work in the setting of Finsler metrics as well. We also note that the result gives higher dimensional examples by embedding the surface as a totally geodesic submanifold. 
			
			\item We say that a homeomorphism $f : (M,g') \rightarrow (M,g)$ is \emph{length shrinking} if 
			\[ d_{g}(f(a), f(b)) \leq d_{g'}(a,b) \text{ for all } a,b \in M.\]
			Although the results and arguments are stated for the case where $f$ is a diffeomorphism, they can be adapted to show that there does not exist a length shrinking homeomorphism as well.
		\end{enumerate}
		
	\end{rem}
	
	We provide a brief description of the proof, with the full proof following in Section \ref{sec:proof}. Fixing a negatively curved metric $g$, let $\gamma$ be a $g$-geodesic with a single self-intersection so that it forms a ``figure eight.'' Suppose that we have constructed a metric $g'$ in such a way that for every length shrinking diffeomorphism \hbox{$f : (M,g') \rightarrow (M,g)$} we have that $\gamma_f \coloneqq f \circ \gamma$ is homotopic to $\gamma$. Indeed, if $\gamma$ is the shortest figure eight $g$-geodesic and has multiplicity one in the length spectrum of $(M,g)$, then $\gamma_f$ is homotopic to $\gamma$ provided $g'$ is sufficiently close to $g$. Furthermore, we suppose that $g'$ is constructed so that the following three properties hold: one loop of $\gamma$ gets shorter by some amount $\xi_1$ while the other gets longer by some amount $\xi_2$ $>\xi_1$, the marked length spectrum of $g'$ is strictly larger than the marked length spectrum of $g$, and $\gamma$ is a $g'$-geodesic after a reparameterization. We also assume that $g'$ has been constructed so that $\xi_2$ can be made arbitrarily close to $\xi_1$ without affecting the above properties. 
	
	With the above setup, we have the ingredients necessary to prove the result. Suppose for contradiction there is a shrinking diffeomorphism $f : (M,g') \rightarrow (M,g)$, and let $p$ be the point of self-intersection for $\gamma$. Using the shrinking property of $f$, we have 
	\[ \ell_g(\gamma) \leq \ell_g(\gamma_f) \leq \ell_{g'}(\gamma) = \ell_g(\gamma) + \xi_2 - \xi_1. \] 
	In particular, this implies that $|\ell_g(\gamma_f) - \ell_g(\gamma)| \leq \xi_2 - \xi_1.$
	If $\xi_2$ is sufficiently close to $\xi_1$, then the fact that closed geodesics which are close in length must be $C^0$-close implies that $\gamma_f$ is $C^0$-close to $\gamma$ (see Lemma ~\ref{lem:lengthc0}). The contradiction now follows from the fact that we can use the (very short) $g$-geodesic connecting $f(p)$ to $p$ along with $\gamma_f$ to construct a curve in the same homotopy class as $\gamma$ which has $g$-length shorter than $\gamma$. Namely, letting $\nu$ be the $g$-geodesic connecting $f(p)$ to $p$ and letting $\gamma_f^i$ be the $i$\textsuperscript{th} loop with $i=1,2$, we have that the concatenated curve $\gamma^2 \nu \gamma_f^1 \nu^{-1}$ has $g$-length less than $\ell_g(\gamma)$. 
	
		The challenge in this proof lies in constructing a metric $g'$ satisfying the properties above. While it is not particularly difficult to construct $g'$ such that for each figure eight, we have that one loop contracts, another loop expands, and the curve is a $g'$-geodesic after a reparameterization, it is difficult to ensure that the marked length spectrum of the new metric is longer than the marked length spectrum of the old metric. Moreover, even after showing that the marked length spectrum is larger, it is difficult to ensure that this bound is multiplicative as opposed to additive. By carefully adjusting the shrinking and expanding parameters and controlling how much the metric expands outside a small neighborhood around the figure-eight, we show in Section \ref{sec:construction} that all of the desired properties hold for this new metric $g'$.
	
	The organization of the paper is as follows. In Section \ref{sec:preliminaries}, we recall some basic facts needed throughout. In Section \ref{sec:outline}, we give a more detailed outline of the proof. In Section \ref{sec:construction}, we construct the metrics satisfying all of the properties outlined above.

	\subsection*{Acknowledgements}
	We would like to acknowledge the anonymous referee for many useful comments, as well as for pointing out a gap in the proof of Theorem \ref{thm:perturb}.
	
	\section{Preliminaries} \label{sec:preliminaries}
	
	Throughout, let $M$ be a closed, connected, orientable surface, let $g$ be a negatively curved Riemannian metric on $M$, and let $p : TM \rightarrow M$ be the footprint map. All geodesics are parameterized with respect to arclength unless otherwise stated. The \emph{connection map} is given by
	\[ K : TSM \rightarrow TM, \quad K_{(x,v)}(\xi) := \frac{D_{p \circ V} V}{dt}(0),\]
	where $V$ is a curve on $SM$ satisfying $V(0) = (x,v)$ and $\dot{V}(0) = \xi$. One can decompose the tangent space of the tangent bundle at $(x,v) \in SM$ using the connection and footprint maps:
	\[ T_{(x,v)}SM = \ker(K_{(x,v)}) \oplus \ker(D_{(x,v)}p).\]
	The \emph{Sasaki metric} $g_{SM}$ is the metric on $TTM$ induced by $g$ which makes these spaces orthogonal:
	\[ (g_{SM})_{(x,v)}(\xi, \xi') := g_x(D_{(x,v)}p(\xi),D_{(x,v)}p(\xi')) + g_x(K_{(x,v)}(\xi),K_{(x,v)}(\xi')).  \]
	Let $\gamma, \eta$ be two smooth $g$-geodesics on $M$. The \emph{$C^1$-distance} between them is given by $d_{C^1}(\gamma, \eta) \coloneqq d^{SM}_{C^0}(\tilde{\gamma}, \tilde{\eta}),$
	where $d^{SM}$ is the metric induced by the Sasaki metric and $\tilde{\gamma}(t) \coloneqq (\gamma(t), \dot{\gamma}(t))$.
	
	The following three lemmas are standard, and will be used throughout. The first lemma is a standard result in the study of geodesics in tubular neighborhoods. We sketch a proof of it for convenience.
	
	\begin{lemma}\label{lem:c1geodesic}
		Let $(M,g)$ be as above and let $\gamma : [0,T] \rightarrow M$ be a closed $g$-geodesic. For every $\epsilon > 0$, there is an open set neighborhood $U_\epsilon \supseteq \text{Im}(\gamma)$ such that if $\eta|_{[0,T]}$ is a $g$-geodesic segment with $\text{Im}(\eta) \subseteq U_\epsilon$, then $d_{C^1}(\gamma, \eta|_{[0,T]}) < \epsilon$. 
	\end{lemma}
	
	\begin{proof}
		Fix a point $t \in [0,T]$ and consider Fermi coordinates around $\gamma(t)$.  In particular, we can take a neighborhood $U_{t,\epsilon}$ around $\gamma(t)$ to be a box centered along $\gamma$ in these Fermi coordinates. Assume that this box is sufficiently small so that it satisfies the \emph{bipoint uniqueness} condition, in the sense that for any two points $p,q \in U_\epsilon$ which are sufficiently close, we have that there is a unique minimal geodesic segment joining $p$ and $q$ contained in $U_\epsilon$. Note that shrinking the width of this box does not change this property.
	
		By shrinking the width of this box, we see that any $g$-geodesic segment which crosses the length has endpoints $C^0$-close to $\gamma$. Using the lemma in \cite[Section 4]{obstacle}, there is a width of the box satisfying the property that any $g$-geodesic segment which crosses the length must be $C^1$-close to $\gamma$ in this box. Using compactness of the interval, we can cover $\text{Im}(\gamma)$ by finitely many $U_{t,\epsilon}$, and taking the union gives us our neighborhood $U_\epsilon$. Moreover, shrinking the widths of these boxes further if needed, we may assume that if $\eta$ crosses the lengths of one of the boxes, then the entirety of $\eta$ is $C^1$-close to $\gamma$. Finally, if $\eta$ is a $g$-geodesic satisfying $\text{Im}(\eta|_{[0,T]}) \subseteq U_\epsilon$, then by construction it is $C^1$-close to $\gamma$.
	\end{proof}
	
	The next lemma is a standard result which allows us to relate lengths and the $C^0$-metric. For details, we point the reader to the proof of Theorem 2 in \cite[Chapter 2]{doCarmo}.
	
	\begin{lemma} \label{lem:lengthc0}
		Let $(M,g)$ be as above, let $g$ be a metric on $M$, and let $\gamma$ be a closed $g$-geodesic. For all $\epsilon > 0$, there is a $\delta > 0$ such that if $\eta$ is a closed curve homotopic to $\gamma$ satisfying $|\ell_g(\eta) - \ell_g(\gamma)| < \delta$, then $d_{C^0}(\eta, \gamma) < \epsilon$.
	\end{lemma}
	
	Finally, given a $g$-geodesic $\gamma$, we are able to smoothly perturb the metric $g$ to get a new metric $g_s$, where a reparameterization of $\gamma$ is still a $g_s$-geodesic. 
	
	\begin{lemma} \label{lem:bumpnbhd}
		Let $M$ be a closed, connected, oriented surface, let $g$ be a negatively curved metric on $M$, let $\gamma$ be a $g$-geodesic on $M$, and let $p \in \gamma$. There exists an open neighborhood $U$ of $p$ such that for every $s$ with $|s| < 1$, every open neighborhood $V$ of $p$ with $V\subseteq U$, and every closed neighborhood $A$ of $p$ with $A \subseteq V$, we can find a smooth bump function $\kappa_s : M \rightarrow \R$ satisfying 
		\begin{itemize}
			\item $(\kappa_s)|_{A} \equiv (1+s)^2$,
			\item $(\kappa_s)|_{V^c} \equiv 1$, where $V^c$ is the complement of $V$,
			\item $g_s := \kappa_s g$ defines a new metric with the property that $\gamma$ defines a $g_s$-geodesic after a reparameterization.
		\end{itemize}
	\end{lemma}
	
	We omit the proof, as it is an easy calculation in Fermi coordinates.
	
	\begin{rem} \label{rem:supp}  As we vary $s$ and $V$ in the last lemma, we get smooth perturbations of $g$. Negative curvature is an open condition, so if we fix the neighborhood $U$ coming from the claim, then there is an $s_0$ so that if $|s| \leq s_0$, then $g_s$ is also a negatively curved metric. Thus, there is an $s_0 > 0$ so that if $|s| \leq s_0$ and $V \subseteq U$, then $g_s$ is a negatively curved metric.
	\end{rem}

	\section{Proof of Theorem \ref{thm:perturb}} \label{sec:proof}
	
	Throughout, let $M$ be a closed, connected, oriented surface and let $g$ be a negatively curved metric on $M$. For notational convenience, let $\MLS(g) \coloneqq \MLS_g$.
	As noted in \cite[Theorem 4.2.4]{buser}, there is at least one shortest $g$-geodesic with a self-intersection, and such a $g$-geodesic will have exactly one self-intersection. Let $\F$ denote the collection of shortest $g$-geodesics with a single self-intersection. Notice that there are finitely many of them, and they all have the same length. For each $\gamma \in \F$, denote the shorter loop by $\gamma^1$ and the longer loop by $\gamma^2$, so $\ell_g(\gamma^1) \leq \ell_g(\gamma^2)$. Let $\gamma_{\short} \in \F$ be such that $\ell_g(\gamma_{\short}^1) \leq \ell_g(\gamma^1$) for all $\gamma \in \F$. In other words, $\gamma_\short$ has the shortest first loop.
	
	\subsection{Outline of the Proof} \label{sec:outline}
	
	The goal is to perturb our metric $g$ so that we get a new metric $g'$ which is $C^\infty$-close to $g$ and such that the following holds:
	\begin{enumerate}
		\item \label{item4} for every shrinking diffeomorphism $f : (M,g') \rightarrow (M,g)$ and $\gamma \in \F$, we have that there is an $\eta \in \F$ so that $f \circ \gamma$ is homotopic to $\eta$,
		\item \label{item1} each $\gamma \in \F$ is a $g'$-geodesic after reparameterization,
		\item \label{item2} there are constants $0 < \xi_{1} < \xi_{2}$ so that for every $\gamma \in \F$ we have 
		\[ \ell_{g'}(\gamma^1) = \ell_g(\gamma^1) - \xi_{1} \text{ and } \ell_{g'}(\gamma^2) = \ell_g(\gamma^2) + \xi_{2},\]
		\item \label{item3} $\MLS(g') > (1+\epsilon)\MLS(g)$. 
		
	\end{enumerate}
	
	Our method of constructing these metrics will also ensure that we are able to further perturb $g$ so that $\xi_2$ is arbitrarily close to $\xi_1$ and the above properties hold.
	
	Given a shrinking diffeomorphism $f : (M,g') \rightarrow (M,g)$, we can use property \eqref{item4} to find an $\eta \in \F$ such that ${\gamma_f \coloneqq f \circ \gamma_\short}$ is homotopic to $\eta$. Notice that we have 
	\[ \ell_g(\eta) = \ell_g(\gamma_\short) \leq \ell_g(\gamma_f) \leq \ell_{g'}(\gamma_\short) = \ell_g(\eta) + \xi_2 - \xi_1.\]
	Using Lemma \ref{lem:lengthc0}, we may assume that $\xi_2$ is close enough to $\xi_1$ so that if $q$ is the point of self-intersection for $\eta$ and $p$ is the point of self-intersection for $\gamma_f$, then for any shrinking diffeomorphism we have $d_g(q, p) < \xi_1/2$. 
	
	Assuming we can construct a metric $g'$ close to $g$ with the above properties, then we have all of the ingredients to show that there cannot be a shrinking diffeomorphism $f : (M,g') \rightarrow (M,g)$, as described in Section \ref{sec:introduction}.
	
	\begin{proof}[Proof of Theorem \ref{thm:perturb}]
		Assume for contradiction that $f : (M,g') \rightarrow (M,g)$ is shrinking. Let $\gamma_f$, $\eta$, $p$ and $q$ be as above. Since $\gamma_f$ is homotopic to $\eta$, we must have that $\gamma_f^1$ is homotopic to either $\eta^1$ or $\eta^2$. Without loss of generality, assume it is homotopic to $\eta^1$. Let $\nu$ be the unique $g$-geodesic connecting $p$ and $q$. By the above discussion, we have $\ell_g(\nu) < \xi_1/2$. Concatenating $\gamma_f^1$ with $\nu$, $\nu^{-1}$, and $\eta^2$, we get a new figure eight curve in the same free homotopy class as $\eta$. Using the length shrinking property, notice that the first loop has length
		\[ \ell_g(\nu^{-1} \gamma_f^1 \nu) < \xi_1 + \ell_g(\gamma_f^1) \leq \xi_1 + \ell_{g'}(\gamma^1_\short) = \ell_g(\gamma^1_\short) \leq \ell_g(\eta^1).  \]
		Thus the curve $\eta^2 \nu^{-1} \gamma^1_f \nu$ has $g$-length smaller than $\eta$, which contradicts the fact that $\eta$ is the curve with the shortest length in its free homotopy class.
	\end{proof}
	
	The main difficulty is constructing the perturbation of $g$ so that the above properties hold. We now describe the intuition behind the construction. Given a curve with one intersection $\gamma \in \F$, we can use Lemma \ref{lem:bumpnbhd} to construct a new metric which is $C^\infty$-close to the original metric and which makes one loop shorter and the other loop longer. We refer to the neighborhood which shrinks a loop of $\gamma$ as the ``shrinking neighborhood'' of $\gamma$, and the neighborhood which expands a loop of $\gamma$ as the ``expanding neighborhood'' of $\gamma$. We also construct the metric in such a way that, outside of a small neighborhood $W$ of $\gamma$, we have that the new metric expands the lengths of curves uniformly compared to the original metric. In this setup, it is clear that any curve which stays outside of $W$ must get longer, so we have an inequality on the marked length spectrum as long as the geodesic does not cross the shrinking neighborhood.
	
	The only problem now is if a geodesic intersects the shrinking neighborhood for $\gamma$, so that a portion of it gets shorter. By making $W$ smaller and using Lemma \ref{lem:c1geodesic}, we can ensure that either this geodesic has to leave the neighborhood $W$, and thus get longer by some amount, or it has to cross the expanding neighborhood for $\gamma$. In either case, after carefully adjusting parameters to ensure that every geodesic gets uniformly longer, we have an inequality on the marked length spectrum and we have the constants $\xi_1$ and $\xi_2$ described above. We will see in Claims \ref{claim:MLS1} and \ref{claim:MLS} that more care has to be taken for the multiplicative bound, but this idea is sufficient for getting an additive bound (see Claim \ref{claim:MLS0}).
	
	By further adjusting the parameters, we can guarantee that the perturbed metric $g'$ is $C^\infty$-close to $g$, and $\xi_2$ is as close to $\xi_1$ as we wish. We will show in Claim \ref{claim:isotopy} that as long as $g'$ and $g$ are sufficiently close, then given any shrinking diffeomorphism $f : (M,g') \rightarrow (M,g)$ and any $\gamma \in \F$, we have that $f \circ \gamma$ is homotopic to some $\eta \in \F$. As long as there is only one curve in $\F$, we will see that we can adjust the parameters so that this metric can be used to prove Theorem \ref{thm:perturb}.
	
	Finally, we will observe that if $\F$ has more than one curve, then we only need to somewhat modify the above construction. This modification, along with the above argument, completes the proof of Theorem \ref{thm:perturb}.
	
	\subsection{Construction of the Metrics} \label{sec:construction}
	
	We start by finding a family of metrics which satisfy property \eqref{item4} from above.
	
	\begin{claim} \label{claim:isotopy}
		There exists a $C^\infty$ neighborhood $U$ of $g$ so that if $g' \in U$, then for every shrinking diffeomorphism $f : (M,g') \rightarrow (M,g)$ and every $\gamma \in \F$ we have that $f \circ \gamma$ is homotopic to some $\eta \in \F$.
	\end{claim}
	
	\begin{proof}
		For $T \in \R$, let 
		\[ P_g(T) = \{ \eta \ | \ \eta \text{ is a } g\text{-geodesic and } \ell_g(\eta) \leq T\}.\] 
		We break this up into a series of steps.
		\begin{enumerate}[$\text{Step }$1:]
			\item Let $\eta_0$ be a $g$-geodesic of shortest length and let $T_0 = \ell_g(\eta_0)$. We claim that as long as $g'$ is sufficiently $C^\infty$-close to $g$, then $f$ preserves $P_g(T_0)$ up to permutation. In other words, if $\eta \in P_g(T_0)$ and $f \circ \eta$ is homotopic to $\eta'$, then $\eta' \in P_g(T_0)$. 
			
			Since $f$ is a shrinking diffeomorphism, we have
			\[ \ell_g(\eta') \leq \ell_g(f \circ \eta) \leq \ell_{g'}(\eta).\]
			Let $\delta > 0$ be such that $P_g(T_0) \subsetneq P_g(T_0 + \delta)$. Notice there are finitely many geodesics in $P_g(T_0 + \delta)$, so if we let 
			\[\Lambda \coloneqq \min\{ \ell_g(\eta') \ | \ \eta' \in P_g(T_0 + \delta) \setminus P_g(T_0)\},  \]
			then by taking $g'$ sufficiently $C^\infty$-close to $g$ we can guarantee that for every shrinking diffeomorphism we have
			\[ \ell_g(\eta') \leq \ell_g(f \circ \eta) \leq \ell_{g'}(\eta) < \Lambda.\]
			Thus we have that $\ell_g(\eta') = T_0$ and so $f$ preserves $P_g(T_0)$ up to permutation.
			
			\item Let $\eta_1$ be a $g$-geodesic of second shortest length and let $T_1 = \ell_g(\eta_1)$. Repeating the argument above, by ensuring that $g'$ is sufficiently close to $g$ we have that every shrinking diffeomorphism preserves $P_g(T_1)$ up to permutation. In particular, intersecting the neighborhood from this step and the last step, we see that every shrinking diffeomorphism must preserve $P_g(T_1) \setminus P_g(T_0)$ up to permutation.
			
			\item Repeating step 2 until we reach $\ell_g(\gamma)$ for $\gamma \in \F$, we have that for all sufficiently close $g'$ if \hbox{$f : (M,g') \rightarrow (M,g)$} is a shrinking diffeomorphism and $\eta$ is homotopic to $f \circ \gamma$, then $\ell_g(\eta) = \ell_g(\gamma)$. Next, notice that for every shrinking diffeomorphism $f : (M,g') \rightarrow (M,g)$ we have
			\[ \ell_g(\gamma) = \ell_g(\eta) \leq \ell_g(f \circ \gamma) \leq \ell_{g'}(\gamma).\]
			By making $g'$ closer to $g$ if needed, we can ensure that $|\ell_g(f \circ \gamma) - \ell_g(\eta)|$ is small for every shrinking diffeomorphism. Using Lemma \ref{lem:lengthc0}, this forces $f \circ \gamma$ to be $C^0$ close to $\eta$. Noting that $f \circ \gamma$ has one self-intersection, we are able to ensure that $\eta$ has a self-intersection by making $g'$ even closer to $g$ if needed, thus $\eta \in \F$. \qedhere   \end{enumerate}
	\end{proof}
	
	
	We now describe the procedure for modifying the metric in the case where $\F$ has only one curve, say $\gamma$. Taking two points $p_1 \in \gamma^1$ and $p_2 \in \gamma^2$, let $U_1$ and $U_2$ be open neighborhoods of $p_1$ and $p_2$ such that Lemma ~\ref{lem:bumpnbhd} applies. Let $W^{\mathrm{i}}$ and $W^{\mathrm{o}}$ be open neighborhoods of $\gamma$ with $\zeta \coloneqq d(\partial W^{\mathrm{i}}, \partial W^{\mathrm{o}}) > 0$, and let $A$ be a closed set satisfying $\im(\gamma)  \subsetneq A \subsetneq W^{\mathrm{i}} \subsetneq W^{\mathrm{o}}$. Consider the following functions.
	\begin{itemize}
		\item For $\epsilon_1 > 0$ and $\rho_1 > 0$ such that $B_{\rho_1}(p_1) \subseteq U_1$, let $\kappa_{-\epsilon_1}$ be the bump function coming from Lemma \ref{lem:bumpnbhd}. Here, $U_1$ corresponds to $U$ in the lemma, and $B_{\rho_1}(p_1)$ corresponds to $V$ in the lemma, and we take some closed box in Fermi coordinates for the neighborhood $A$ in the lemma.
	%
		\item For $\epsilon_2 > 0$ and $\rho_2 > 0$ such that $B_{\rho_2}(p_2) \subseteq U_2$, let $\kappa_{\epsilon_2}$ be the bump function coming from Lemma ~\ref{lem:bumpnbhd}. As above, $U_2$ corresponds to $U$ in the lemma, and $B_{\rho_2}(p_2)$ corresponds to $V$ in the lemma, and we take some closed box in Fermi coordinates for the neighborhood $A$ in the lemma.
	%
		\item For $\epsilon_3 > 0$, let $\kappa_{\epsilon_3}$ be a bump function on $M$, where
		\[ (\kappa_{\epsilon_3})|_{(W^{\mathrm{i}})^c} \equiv (1+\epsilon_3)^2 \text{ and } (\kappa_{\epsilon_3})|_{A} \equiv 1.\]
	\end{itemize}
	We define a family of metrics by setting
	\[ g_{W^{\mathrm{i}}, W^{\mathrm{o}}, A, \rho_1, \rho_2, \epsilon_1, \epsilon_2, \epsilon_3} := \kappa_{\epsilon_3} \kappa_{\epsilon_2} \kappa_{-\epsilon_1} g.\]
	Notice that $g$ does not really depend on $W^{\mathrm{o}}$ beyond the fact that $W^{\mathrm{i}} \subseteq W^{\mathrm{o}}$ and $d(\partial W^{\mathrm{i}}, \partial W^{\mathrm{o}}) > 0$, but we include it for convenience, as this will become a parameter that we will vary in the future. Furthermore, we note that we will only consider those $W^{\mathrm{o}}$ which come from shrinking the widths of the boxes given in Lemma \ref{lem:c1geodesic}, and thus we may assume that $W^{\mathrm{o}}$ satisfies the bipoint uniqueness condition. This gives us a smooth family of metrics parameterized by $W^{\mathrm{i}}$, $W^{\mathrm{o}}$, $A$, $\rho_i$,  and $\epsilon_j$ which all satisfy property \ref{item1}. At this point, it is easy to adjust the constants to get property \ref{item2}, but a more careful analysis is needed to guarantee property \ref{item3}. We explore this now.

	We start by fixing some $ W^{\mathrm{i}} \subsetneq W^{\mathrm{o}}$ open and some $A$ closed satisfying the above criteria, along with some $\rho_1, \rho_2 > 0$. As noted in Remark \ref{rem:supp}, there exists uniform upper bounds $t_1 > 0$ so that for all $\epsilon_j$ satisfying $\epsilon_j < t_1$ we have that the metrics $g_{W^{\mathrm{i}}, W^{\mathrm{o}},A,\rho_1, \rho_2,\epsilon_1, \epsilon_2, \epsilon_3}$ are negatively curved. We can also choose $t_1$ sufficiently small so that we can apply Claim ~\ref{claim:isotopy} with all metrics satisfying this condition. Note that we are free to shrink the parameters $ W^{\mathrm{i}}$, $W^{\mathrm{o}}$ $A$, and $\rho_i$ as we wish, and this property will still hold. From now on, we only consider $\epsilon_j$ satisfying $\epsilon_j < t_1$.

	\begin{figure}[H] 
		\begin{center}
			\begin{tikzpicture}[scale=0.8]
				
				
				\draw[black,  domain=-3:5, samples=100] 
				plot({4*cos(\x r)/(1+sin(\x r)^2)},{(6 * sin(\x r) * cos(\x r))/(1 + sin(\x r)^2)});
				
				
				\node[label = left:{\tiny $p_1$}] at (-2,-2) [circle,fill,inner sep=1.5pt] {};
				\draw[dotted, thick, green, fill=green, opacity = 0.25] (-2,-2) circle (0.4cm);
				
				\node[label = right:{\tiny $p_2$}] at (2,-2) [circle,fill,inner sep=1.5pt]{};
				\draw[dotted, thick, green, fill=green, opacity = 0.25] (2,-2) circle (0.4cm);
				
				\draw[orange,dotted,thick] (2.25,0) circle (0.5cm);
				\begin{scope}[yscale=1,xscale=-1]
					\draw[orange,dotted,thick] (0,2.5) arc (45:315:3.6cm);
				\end{scope}
				
				\draw[orange,dotted,thick] (-2.25,0) circle (0.5cm);
				\draw[orange,dotted,thick] (0,2.5) arc (45:315:3.6cm);

				\draw[blue,dotted,thick] (2.25,0) circle (1cm);
				\begin{scope}[yscale=1,xscale=-1]
					\draw[blue,dotted,thick] (0,2) arc (45:315:3cm);
				\end{scope}
				
				\draw[blue,dotted,thick] (-2.25,0) circle (1cm);
				\draw[blue,dotted,thick] (0,2) arc (45:315:3cm);
				
				\draw[name = A1, red,thick] (2.25,0) circle (1.5cm);
				
				\begin{scope}[yscale=1,xscale=-1]
					\draw[name = A2, red,thick] (0,1.5) arc (40:320:2.5cm);
				\end{scope}

				\draw[red,thick] (-2.25,0) circle (1.5cm);
				\draw[red,thick] (0,1.5) arc (40:320:2.5cm);
			\end{tikzpicture}
			\caption{The orange dotted lines represent the boundary of the neighborhood $W^{\mathrm{o}}$, the blue dotted lines represent the boundary of the neighborhood $W^{\mathrm{i}}$, the red lines represent the boundary of the closed neighborhood $A$, and the green disks represent the sets $B_{\rho_i}(p_i)$.}\label{fig:nbhd}
		\end{center}

	\end{figure}
	
	Since we no longer need to adjust $\rho_1$, $\rho_2$, and $\epsilon_3$, we fix them once and for all. All modifications to $W^{\mathrm{i}}$ and $A$ will be done relative to $W^{\mathrm{o}}$, and so we omit these for notational convenience. Thus, we will now consider the family of metrics $g_{W^{\mathrm{o}}, \epsilon_1, \epsilon_2}$ where the parameters satisfy the properties above. 
	
	The goal is to now study how to modify the parameters so that we have $\MLS(g_{W^{\mathrm{o}}, \epsilon_1, \epsilon_2}) > (1+\epsilon)\MLS(g)$ for some $\epsilon > 0$. To start, let $\tilde{O}_{\epsilon_1} \coloneqq \supp(\kappa_{-\epsilon_1}-1)$, $\tilde{O}_{\epsilon_2} \coloneqq \supp(\kappa_{\epsilon_2}-1)$,
	\[ \xi_1 := \int_{\gamma^{-1}(\tilde{O}_{\epsilon_1})} \|\dot{\gamma}(t)\|dt - \int_{\gamma^{-1}(\tilde{O}_{\epsilon_1})} \|\dot{\gamma}(t)\|_{W^{\mathrm{o}},\epsilon_1, \epsilon_2, \epsilon_3}dt,\]
	and
	\[ \xi_2 \coloneqq \int_{\gamma^{-1}(\tilde{O}_{\epsilon_2})} \|\dot{\gamma}(t)\|_{W^{\mathrm{o}},\epsilon_1, \epsilon_2, \epsilon_3}dt -\int_{\gamma^{-1}(\tilde{O}_{\epsilon_2})} \|\dot{\gamma}(t)\|dt  .\]
	The quantities $\xi_1$ and $\xi_2$ will be used to measure the expansion and contraction of a geodesic segment relative to $\gamma$. Observe that since $W^{\mathrm{o}}$ must contain the curve $\gamma$, the quantities $\xi_1$ and $\xi_2$ are not influenced by $W^{\mathrm{o}}$. Finally, let $L \coloneqq \ell_{g_{W^{\mathrm{o}},\epsilon_1, \epsilon_2}}(\gamma)$.

	Observe that if we assume that $\epsilon_1$ is sufficiently small relative to $\epsilon_3$, then the only way a curve can possibly get shorter in the $g_{W^{\mathrm{o}}, \epsilon_1, \epsilon_2}$-metric compared to the $g$ metric is if it intersects the neighborhood $B_{\rho_1}(p_1) \cap W^{\mathrm{i}}$. We now refer to this open set as the ``shrinking neighborhood.'' Let $t_2$ be such that for $\epsilon_1 < t_2$ the above holds, and let $\mu = \epsilon_3 - \epsilon_1$. Notice that if a curve $\eta$ intersects $W^{\mathrm{i}}$ and leaves $W^{\mathrm{o}}$, then it must cross $W^{\mathrm{o}} \setminus W^{\mathrm{i}}$ and hence get longer by at least $\zeta \mu$, where again $\zeta = d_g(\partial W^{\mathrm{i}}, \partial W^{\mathrm{o}})$. This leads us to our first observation, which is an additive bound on the marked length spectrum.

\begin{claim} \label{claim:MLS0}
	Assume that $\epsilon_1, \epsilon_2 < t_2$. There is $t_3 > 0$ and an open set $W_1 \supseteq \text{Im}(\gamma)$ such that if $W^{\mathrm{o}} \subseteq W_1$, $\epsilon_1 < t_3$, and $\xi_1 + s < \xi_2$ for some $s$ depending on $W^{\mathrm{o}}$, then there is a uniform $\lambda > 0$ so that $\MLS(g_{W^{\mathrm{o}}, \epsilon_1, \epsilon_2}) > \MLS(g) + \lambda$.
\end{claim}

\begin{proof}
	Let $\eta$ be a closed $g_{W^{\mathrm{o}}, \epsilon_1, \epsilon_2}$-geodesic. If $\eta$ does not intersect $W^{\mathrm{i}}$, then it is clear that it gets longer by some factor related to its length and $\mu$. Taking the minimum over the lengths of all closed geodesics gives us a uniform lower bound $\lambda_1$ so that 
	\begin{equation} \label{eqn:abound1} \ell_{g_{W^{\mathrm{o}}, \epsilon_1, \epsilon_2}}(\eta) > \ell_g(\eta) + \lambda_1. \end{equation}
	
	Next, suppose that $\text{Im}(\eta)$ intersects $W^{\mathrm{i}}$. If the curve is not entirely contained within $W^{\mathrm{o}}$, then it must get longer by at least $2 \zeta \mu$, as it has to leave and re-enter $W^{\mathrm{o}}$. If it does not intersect the shrinking neighborhood, then this gives us the lower bound
	\begin{equation} \label{eqn:abound2} \ell_{g_{W^{\mathrm{o}}, \epsilon_1, \epsilon_2}}(\eta) > \ell_g(\eta) + 2 \zeta \mu. \end{equation}
	Finally, suppose that $\text{Im}(\eta)$ intersects the shrinking neighborhood $B_{\rho_1}(p_1) \cap W^{\mathrm{o}}$. Notice that if $\eta$ intersects the shrinking neighborhood, then we have that the curve must get shorter by at least $\xi_1 + \delta_1$, where $\delta_1$ is some uniform constant depending on $W^{\mathrm{o}}$ that tends to zero as $W^{\mathrm{o}}$ gets smaller. The parameter $\delta_1$ arises from the fact that we do not have any estimates on how long $\eta$ spends in the shrinking neighborhood -- it could possibly spend more time than $\gamma$, and hence shrink more than $\gamma$ does. By compactness of the unit tangent bundle, there is a curve that spends the most amount of time in this neighborhood, and so this $\delta_1$ is realized through this curve. Furthermore, the amount of time that it spends in the shrinking neighborhood is clearly controlled by $W^{\mathrm{o}}$, and hence making $W^{\mathrm{o}}$ smaller will make $\delta_1$ smaller.
	
	After adjusting $W^{\mathrm{o}}$ and using Lemma \ref{lem:c1geodesic}, we see that once $\eta$ intersects the shrinking neighborhood, there are two options: either $\eta$ must stay $C^1$-close to $\gamma$ until it crosses $B_{\rho_2}(p_2) \cap W^{\mathrm{o}}$, and hence $\eta$ will get longer by at least $\xi_2 - \delta_2$, where $\delta_2$ depends on $W^{\mathrm{o}}$ and tends to zero as $W^{\mathrm{o}}$ gets smaller, or $\eta$ must leave and re-enter $W^{\mathrm{o}}$, and hence $\eta$ will get longer by at least $2\zeta \mu$. In either case, if $\eta$ crosses the shrinking neighborhood $m$ times with $m \geq 1$, then setting $s \coloneqq \delta_1 + \delta_2$ we have the bound
	\begin{equation} \label{eqn:abound3} \ell_{g_{W^{\mathrm{o}}, \epsilon_1, \epsilon_2}}(\eta) > \ell_g(\eta) + m \cdot \min\{(\xi_2 - \xi_1 - s), 2\zeta \mu - (\xi_1 + \delta_1)\}. \end{equation}
	Note that this bound also accounts for the case where $\eta$ is entirely contained within $W^{\mathrm{i}}$, and this exhausts the last possibility for $\eta$. Ensuring $\epsilon_1$ and $W^{\mathrm{o}}$ are small enough, we have that the bounds in \eqref{eqn:abound1}, \eqref{eqn:abound2}, and \eqref{eqn:abound3} are all positive. Letting $\lambda$ be the minimum of these bounds yields the desired result.
\end{proof}

\begin{rem}
If we replace the multiplicative bound in Theorem \ref{thm:perturb} with an additive bound, then we now have the ingredients to prove this modified version of the theorem via the argument outlined in Section \ref{sec:outline}.
\end{rem}

The previous argument suggests that a multiplicative bound is possible, and we make this rigorous with the following.

\begin{claim} \label{claim:MLS1}
		Assume that $\epsilon_1, \epsilon_2 < t_3$ and $W^{\mathrm{o}} \subseteq W_1$. 
		There is a $t_4 > 0$ and an open set $W_2$ with $W_1 \supseteq W_2 \supseteq \text{Im}(\gamma)$ such that if $W^{\mathrm{o}} \subseteq W_2$, $\epsilon_1 < t_4$, and $\xi_1 + s'' < \xi_2$ for some $s$ depending on $W^{\mathrm{o}}$, then the following hold.
		\begin{enumerate}[(a)]
			\item \label{item:claim1} There is a uniform $\epsilon'' > 0$ so that if $\eta$ is a closed $g_{ W^{\mathrm{o}},  \epsilon_1, \epsilon_2}$-geodesic with $\ell_{g_{W^{\mathrm{o}},\epsilon_1, \epsilon_2}}(\eta) \geq N L$ and $N \geq 1$, then
			\[ 
			\ell_{g_{ W^{\mathrm{o}},  \epsilon_1, \epsilon_2}}\left(\eta |_{[0,NL]} \right) > \ell_g\left(\eta |_{[0,NL]} \right) + N\epsilon''. 
			\]
			\item \label{item:claim2} There is a uniform $\epsilon' > 0$ and an $N_0 > 0$ so that for $N \geq N_0$, if $\eta$ is a closed $g_{ W^{\mathrm{o}},  \epsilon_1, \epsilon_2}$-geodesic satisfying
			\[ NL \leq \ell_{g_{ W^{\mathrm{o}},  \epsilon_1, \epsilon_2}}(\eta) < (N+1)L,\]
			then
			\[ \ell_{g_{ W^{\mathrm{o}},  \epsilon_1, \epsilon_2}}(\eta) > \ell_{g}(\eta) + N \epsilon'.\]
		\end{enumerate} 
	\end{claim} 
	
	\begin{proof}
		We first prove the case $N=1$ for \eqref{item:claim1}. Note that the argument is similar to that of Claim \ref{claim:MLS0}. Indeed, replacing $\eta$ with $\eta|_{[0,L]}$ and $\lambda_1$ with $\mu L$, the bounds in \eqref{eqn:abound1} and \eqref{eqn:abound2} follow. We note that, after adjusting $W^{\mathrm{o}}$ and using Lemma \ref{lem:c1geodesic}, we only need to consider what happens $\eta|_{[0,L]}$ intersects a shrinking neighborhood. Some care needs to be taken here as opposed to Claim \ref{claim:MLS0}, since $\eta|_{[0,L]}$ can intersect the shrinking neighborhood an additional time in this scenario. However, as we will see, additional intersections of the shrinking neighborhood can only occur if the curve $\eta|_{[0,L]}$ leaves $W^{\mathrm{o}}$ inbetween, and this allows for us to let $\xi_2$ be as close to $\xi_1$ as we wish. 
		
		We break up the argument into cases assuming that $\eta|_{[0,L]}$ intersects the shrinking neighborhood.
		\begin{enumerate}[$\text{Case}$ 1:]
			\item If $\eta|_{[0,L]}$ does not start nor end in the shrinking neighborhood, then either it is completely contained in $W^{\mathrm{o}}$ or it leaves $W^{\mathrm{o}}$. We consider these separately.
			\begin{enumerate}[$\text{Subcase}$ 1:]
				\item If $\eta|_{[0,L]}$ stays entirely in $W^{\mathrm{o}}$, then, after appropriately adjusting $W^{\mathrm{o}}$, it must be $C^1$-close to $\gamma$, and hence it must cross $W^{\mathrm{o}} \cap B_{\rho_2}(p_2)$. In this case, we see that the bound in \eqref{eqn:abound3} applies with $m = 1$.
				
				\item If $\eta|_{[0,L]}$ does not stay entirely in $W^{\mathrm{o}}$, then it is possible that $\eta|_{[0,L]}$ intersects the shrinking neighborhood more than once. In particular, we see that it must leave $W^{\mathrm{o}}$ inbetween each intersection of the shrinking neighborhood, as $\eta|_{[0,L]}$ is not allowed to ``turn around'' inside of $W^{\mathrm{o}}$ due to the bipoint uniqueness condition. We also note that $\eta|_{[0,L]}$ could intersect the shrinking neighborhood one last time before reaching its endpoint. Assuming that it intersects the shrinking neighborhood $m$ times, with $m \geq 1$, we overcompensate for this additional intersection and consider the lower bound
				\begin{equation} \label{eqn:abound4} \ell_{g_{W^{\mathrm{o}}, \epsilon_1, \epsilon_2}}(\eta|_{[0,L]}) > \ell_g(\eta|_{[0,L]}) + m \cdot [2\zeta \mu - (\xi_1 + \delta_1)] - (\xi_1 + \delta_1). \end{equation}
				Finally, it is also possible that, after intersecting the shrinking neighborhood $m$ times, $\eta|_{[0,L]}$ leaves $W^{\mathrm{o}}$ and does not return. This yields the lower bound 
				\begin{equation} \label{eqn:abound5} \ell_{g_{W^{\mathrm{o}}, \epsilon_1, \epsilon_2}}(\eta|_{[0,L]}) > \ell_g(\eta|_{[0,L]}) + (m-1) \cdot [2\zeta \mu - (\xi_1 + \delta_1)] + \zeta \mu - (\xi_1 + \delta_1) . \end{equation}
				Note that this bound also accounts for $\eta|_{[0,L]}$ starting outside of $W^{\mathrm{o}}$ and entering the set in order to intersect with the shrinking neighborhood.
			\end{enumerate}
			
			\item If $\eta|_{[0,L]}$ either starts or ends in the shrinking neighborhood (but not both), then the same sort of arguments as in the previous case apply, and we omit them.
			
			\item If $\eta|_{[0,L]}$ starts and ends in the shrinking neighborhood, then we see that either $\eta_{[0,L]}$ stays entirely in $W^{\mathrm{o}}$ or it leaves. We list these subcases separately.
			\begin{enumerate}[$\text{Subcase}$ 1:]
				\item 	If $\eta|_{[0,L]}$ stays entirely in $W^{\mathrm{o}}$, then Lemma \ref{lem:c1geodesic}, implies that we have $\eta|_{[0,L]}$ is $C^1$-close to $\gamma$, hence it must cross the expanding neighborhood as well. Using compactness of the unit tangent bundle again, we see that there is an $s'$ depending on $W^{\mathrm{o}}$ so that 
				\begin{equation} \label{eqn:abound6} \ell_{g_{W^{\mathrm{o}}, \epsilon_1, \epsilon_2}}(\eta|_{[0,L]}) > \ell_g(\eta|_{[0,L]}) + (\xi_2 - \xi_1 - s'). \end{equation}
				Note that this $s'$ does not necessarily agree with the $s$ from Claim \ref{claim:MLS0}, as it arises from considering the two intersections with the shrinking neighborhood. Still, this $s'$ depends on $W^{\mathrm{o}}$, and it tends to zero as $W^{\mathrm{o}}$ gets smaller.
				
				\item  If $\eta|_{[0,L]}$ does not stay entirely in $W^{\mathrm{o}}$, then the same argument in Subcase 2 of Case 1 applies, and thus we have the same bounds.
			\end{enumerate}
		\end{enumerate}

		Adjusting $\epsilon_1$ and $W^{\mathrm{o}}$ appropriately, we have that the bounds in \eqref{eqn:abound1}, \eqref{eqn:abound2}, \eqref{eqn:abound3}, \eqref{eqn:abound4}, \eqref{eqn:abound5}, and \eqref{eqn:abound6} are all positive. We set $\epsilon''$ to be the minimum of the quantities on the right, and we set $s''$ to be the maximum between $s$ and $s'$.
		For $N > 1$, we can apply the argument above on each segment $\eta|_{[(N-1)L, NL]}$, here using the fact that the constants $\epsilon''$ and $s''$ are independent from the choice of geodesic. This finishes the proof of \eqref{item:claim1}.

		We now prove \eqref{item:claim2}. Suppose that $NL \leq \ell_{g_{W^{\mathrm{o}}, \epsilon_1, \epsilon_2}}(\eta) < (N+1)L$ for some $N \geq 1$. Let $S \coloneqq \ell_{g_{W^{\mathrm{o}},\epsilon_1, \epsilon_2}}(\eta)$. By \eqref{item:claim1}, we have that $\ell_{g_{W^{\mathrm{o}}, \epsilon_1, \epsilon_2}}(\eta|_{[0,NL]}) > \ell_g(\eta|_{[0,NL]}) + N \epsilon''$. Notice that, at worst, we have that $\eta|_{[NL, S]}$ remains entirely in the shrinking neighborhood, so there is some $\theta = \theta(L, \epsilon_1, W^{\mathrm{o}})$ so that \hbox{$\ell_{g_{W^{\mathrm{o}}, \epsilon_1, \epsilon_2}}(\eta) > \ell_g(\eta) + N \epsilon'' - \theta$}. Note that $\theta$ tends to $0$ as $\epsilon_1$ tends to zero and is independent of $N$, so in particular we can take a uniform lower bound $N\epsilon'' - \theta_0$ which applies for all $\epsilon_1 < t_4$ and all $W^{\mathrm{o}} \subseteq W_2$. Since $\epsilon''$ and $\theta_0$ are now fixed, take $N_0 > \theta_0/\epsilon''$ and let $\epsilon' \coloneqq 1 - N_0 \epsilon'' / (N_0+1)$. Since $N\epsilon' \leq (N - N_0) \epsilon''$ by construction, we have that
		\[\begin{split} \ell_{g_{W^{\mathrm{o}}, \epsilon_1, \epsilon_2}}(\eta) > \ell_g(\eta) + N \epsilon'' - \theta_0 & > \ell_g(\eta) + (N-N_0) \epsilon'' + N_0 \epsilon'' - \theta_0 \\ 
		& > \ell_g(\eta) +  N \epsilon',\end{split}\]
		as desired. \qedhere
		
		
	\end{proof}
		
	We now have the ingredients to prove the multiplicative bound on the marked length spectrum.
	
	\begin{claim} \label{claim:MLS}
		Assume that $\epsilon_1, \epsilon_2 < t_4$ and $W^{\mathrm{o}} \subseteq W_4$. If $\xi_2 > \xi_1 + s$ for some $s$ depending on $W^{\mathrm{o}}$, then there is an $\epsilon > 0$ so that
		\[ \MLS(g_{W^{\mathrm{o}}, \epsilon_1, \epsilon_2} ) > (1+\epsilon)\MLS(g). \]
	\end{claim}

	\begin{proof}
		Let $\epsilon'$, $\lambda$ $N_0$,and $L$ be from Claims \ref{claim:MLS0} and \ref{claim:MLS1} \eqref{item:claim2}. Let $\epsilon > 0$ be such that $\epsilon' > 2L \epsilon$ and $\lambda > N_0 L \epsilon$.
		Let $\eta$ be a closed $g_{W^{\mathrm{o}},\epsilon_1, \epsilon_2}$-geodesic. If $\ell_{g_{W^{\mathrm{o}}, \epsilon_1, \epsilon_2}}(\eta) \in [NL, (N+1)L)$ for some $N \geq N_0$, then we have 
		\[ 
		\begin{split}\ell_{g_{W^{\mathrm{o}}, \epsilon_1, \epsilon_2}}(\eta) > N \epsilon' + \ell_g(\eta) &> N L\left(\frac{N+1}{N}\right) \epsilon  + \ell_g(\eta) \\
			& > (1+\epsilon)\ell_g(\eta).\end{split}\]
		Next, notice that the argument in Claim \ref{claim:MLS0} implies that $N_0 L \geq \ell_{g_{W^{\mathrm{o}}, \epsilon_1, \epsilon_2}}(\eta) > \ell_g(\eta) + \lambda$. Thus, if we have $\ell_{g_{W^{\mathrm{o}}, \epsilon_1, \epsilon_2}}(\eta) \leq N_0 L$, then 
		\[\frac{\ell_{g_{W^{\mathrm{o}}, \epsilon_1, \epsilon_2}}(\eta)}{\ell_g(\eta)} > 1 + \frac{\lambda}{\ell_g(\eta)} > 1 + \frac{\lambda}{N_0L} > 1 + \epsilon. \]
		The result follows.
	\end{proof}
	
	\begin{rem}
		Fixing $\epsilon_1$ above, we can let $\xi_2$ be arbitrarily close to $\xi_1$ by adjusting $W^{\mathrm{o}}$. Recall that this is necessary for the proof outlined in Sections \ref{sec:introduction} and \ref{sec:outline}.
	\end{rem}
	
	As described in Section \ref{sec:outline}, this gives us the ingredients to prove the theorem provided there is only one curve in $\F$. If there is more than one curve in $\F$, then we apply the same construction in a neighborhood of each $\gamma \in \F$. If the curves do not intersect, the results and arguments are almost the same provided we choose the parameters for the curves so that $\xi_1$ and $\xi_2$ are uniform among all curves. Note that $\epsilon_3$ will be chosen uniformly for all curves, and the neighborhoods $A$, $W^{\mathrm{i}}$ and $W^{\mathrm{o}}$ will be the union of all of the corresponding neighborhoods for each curve. The only adjustment that needs to be made in the case of intersecting curves is that the shrinking and expanding neighborhoods must be chosen so that they do not overlap with any shrinking and expanding neighborhood, and are away from the points of intersection. These adjustments, along with the argument in Section \ref{sec:outline}, prove Theorem \ref{thm:perturb}.
	
	\bibliographystyle{plain}
	\bibliography{bibliography.bib}

\end{document}